\documentclass[12pt]{amsart}
\usepackage{amssymb,latexsym,amsmath,epsfig,amsthm,mathrsfs}
\usepackage{amsfonts}
\usepackage{amscd}
\usepackage{dsfont}
\usepackage{bbm}
\usepackage{rotating}
\usepackage{graphicx}
\usepackage{amssymb}
\usepackage{lineno}
\usepackage{enumitem}
\usepackage[top=3cm, bottom=3cm, left=3.5cm, right=3.5cm]{geometry}
\usepackage[usenames]{color}
\usepackage[colorlinks=true,
linkcolor=blue,
filecolor=blue,
citecolor=blue]{hyperref}

\newtheorem{theorem}{Theorem}[section]

\theoremstyle{definition}

\newtheorem{conjecture}[theorem]{Conjecture}

\theoremstyle{remark}

\numberwithin{equation}{section}
\numberwithin{theorem}{section}

\begin{document}

\title[Restricted signed sumsets in integers]{Direct and inverse results on restricted signed sumsets in integers}

\author[J Bhanja]{Jagannath Bhanja$^{*}$}
\address{Department of Mathematics, Indian Institute of Technology Roorkee, Uttarakhand, 247667, India}
\email{jbhanja90@gmail.com}
\thanks{$^{*}$Research supported by the Ministry of Human Resource Development, India.}
\author[T Komatsu]{Takao Komatsu}
\address{School of Mathematics and Statistics, Wuhan University, Wuhan, 430072, China}
\email{komatsu@whu.edu.cn}
\author[R K Pandey]{Ram Krishna Pandey}
\address{Department of Mathematics, Indian Institute of Technology Roorkee, Uttarakhand, 247667, India}
\email{ramkpandey@gmail.com}





\begin{abstract}
Let $G$ be an additive abelian group. Let $A=\{a_0, a_1,\ldots, a_{k-1}\}$ be a nonempty finite subset of $G$. For a positive integer $h$ satisfying $1\leq h\leq k$, we let
\[
h\hat{}_{\underline{+}}A:=\{\Sigma_{i=0}^{k-1}\lambda_{i} a_{i}: (\lambda_{0},\lambda_{1}, \ldots, \lambda_{k-1}) \in \{-1,0,1\}^{k},~\Sigma_{i=0}^{k-1}|\lambda_{i}|=h \},
\]
be the {\it restricted signed sumset} of $A$.

The {\it direct problem} for the restricted signed sumset $h\hat{}_{\underline{+}}A$ is to find the minimum number of elements in $h\hat{}_{\underline{+}}A$ in terms of $|A|$. The {\it inverse problem} for $h\hat{}_{\underline{+}}A$ is to determine the structure of the finite set $A$ for which $|h\hat{}_{\underline{+}}A|$ is minimal. In this article, we solve some cases of both direct and inverse problems for $h\hat{}_{\underline{+}}A$, when $A$ is a finite set of integers. In this connection, we also pose some questions as conjectures in the remaining cases.
\end{abstract}

\maketitle


\section{Introduction}\label{sec1}

Let $G$ be an additive abelian group. Let $A=\{a_0, a_1,\ldots, a_{k-1}\}$ be a nonempty finite subset of $G$. Let $h$ be a positive integer. We let $hA$, $h\hat{} A$ and $h_{\underline{+}}A$ be the $h$-fold {\it sumset}, the $h$-fold {\it restricted sumset} and the $h$-fold {\it signed sumset} of $A$, respectively (see \cite{N96, BM15, BM2015}); that is,
\[
hA=\{\Sigma_{i=0}^{k-1}\lambda_{i} a_{i}: (\lambda_{0},\lambda_{1}, \ldots,\lambda_{k-1}) \in \mathbb{N}_{0}^{k}, \Sigma_{i=0}^{k-1}\lambda_{i}=h\},
\]

\[
h\hat{} A=\{\Sigma_{i=0}^{k-1}\lambda_{i} a_{i}:~(\lambda_{0},\lambda_{1}, \ldots, \lambda_{k-1}) \in \{0,1\}^k,\Sigma_{i=0}^{k-1} \lambda_{i}=h\},
\]
and
\[
h_{\underline{+}}A=\{\Sigma_{i=0}^{k-1}\lambda_{i} a_{i}: (\lambda_{0},\lambda_{1}, \ldots, \lambda_{k-1}) \in \mathbb{Z}^{k}, \Sigma_{i=0}^{k-1}|\lambda_{i}|=h \},
\]
where $1 \leq h \leq k$ in case of $h\hat{} A$.

Analogous to the signed sumset $h_{\underline{+}}A$, we define the $h$-fold {\it restricted signed sumset} of $A$ for $1\leq h\leq k$, denoted by $h\hat{}_{\underline{+}}A$, by
\[
h\hat{}_{\underline{+}}A:=\{\Sigma_{i=0}^{k-1} \lambda_{i} a_{i}:~(\lambda_{0},\lambda_{1}, \ldots, \lambda_{k-1}) \in \{-1,0,1\}^k,~\Sigma_{i=0}^{k-1} |\lambda_{i}|=h\}.
\]

Clearly,
\begin{equation*}
h\hat{} A\cup h\hat{} (-A) \subseteq h\hat{}_{\underline{+}}A.
\end{equation*}
Also, for an integer $\alpha$, we have
\begin{equation*}
h\hat{}_{\underline{+}} (\alpha *A)=\alpha * (h\hat{}_{\underline{+}}A),
\end{equation*}
where $\alpha*A=\{\alpha \cdot a ~ |~a\in A\}$ is the {\it $\alpha$-dilation} of the set $A$.

The study of sumsets of sets of an additive abelian group has more than two-hundred-year old history. A paper of Cauchy \cite{cauchy} in 1813, which is believed to be one of the oldest and classical work off-course, finds the minimum cardinality of the sumset $A+B$, where $A$ and $B$ are nonempty subsets of residue classes modulo a prime. Later, Davenport \cite{dav} rediscovered Cauchy's result in 1935. The result is now known as the Cauchy-Davenport theorem:
\begin{theorem} [Cauchy-Davenport Theorem]
Let $A$ and $B$ be nonempty subsets of the group $\mathbb{Z}_p$ of prime order $p$. Then
\[
|A+B|\geq \min\{p,|A|+|B|-1\}.
\]
\end{theorem}
The multiple fold generalization of this theorem is the following:
\begin{theorem}
Let $A$ be a nonempty subset of the group $\mathbb{Z}_p$ of prime order $p$. Then
\[
|hA|\geq \min\{p,h|A|-h+1\}.
\]
\end{theorem}

Several partial results about the minimum cardinality of the sumsets and its inverse that if the minimum cardinality is achieved, then the characterization of individual sets have been obtained in the past. A comprehensive list of references may be found in Mann \cite{mann}, Freiman \cite{freiman}, Nathanson \cite{N96}, and  Tao \cite{tao}. Plagne \cite{plagne} (see also \cite{EKP}) settled the general case by obtaining the minimum cardinality of sumset in an abelian group. Theorem of Plagne is given below.

\begin{theorem}[Plagne]
Let $G$ be an abelian group of order $n$. Let $A$ be a nonempty subset of $G$ with cardinality $k$. Then
\[
|hA|\geq \min \{(h\lceil k/d\rceil-h+1)\cdot d:~d\in D(n)\},
\]
where, $D(n)$ is the set of positive divisors of $n$.
\end{theorem}

In contrast, the $h$-fold restricted sumsets are not well-settled. In case of group of all integers, the minimum size of the restricted sumset was given by Nathanson \cite{nathu} in the following theorem.

\begin{theorem} \label{thmA}
Let $A$ be a finite set of $k$ integers, and let $1\leq h\leq k$ be a positive integer. Then
\[
|h\hat{}A| \geq hk-h^2+1.
\]
\end{theorem}

Nathanson \cite{nathu} also classified the sets of integers which give the exact lower bound.

\begin{theorem} \label{thmB}
Let $2\leq h\leq k-2$. Let $A$ be a finite set of $k$ $(\geq 5)$ integers. Then
$|h\hat{}A|=hk-h^2+1$ if and only if $A$ is a $k$-term arithmetic progression.
\end{theorem}

In case of finite abelian groups, the only complete result for the minimum size of the restricted sumset known to us is the Erd\H{o}s-Heilbronn Theorem, which was actually a conjecture by Erd\H{o}s and Heilbronn \cite{EH} in 1964, until it was first confirmed by Dias-Da Silva and Hamidoune \cite{SH} in 1994 using some ideas from the exterior algebra. Later, it was reproved by Alon, Nathanson and Ruzsa \cite{ANR1, ANR2} using the polynomial method.
\begin{theorem} [Erd\H{o}s-Heilbronn Theorem]
Let $A$ be a nonempty subset of the group $\mathbb{Z}_p$ of prime order $p$. Then
\[
|h\hat{}A|\geq \min \{p,h|A|-h^2+1\}.
\]
\end{theorem}

Finding the minimum size of restricted sumsets for general finite abelian groups seems to be more difficult problem than the usual $h$-fold sumsets as the minimum size of restricted sumsets heavily depends on the structure of the group rather than its size.

Turning now to the $h$-fold signed sumsets, $h_{\underline{+}}A$ has a brief and a quite new history. This sumset first appeared in the work of Bajnok and Ruzsa \cite{BR03} in the context of the ``independence number" of a subset $A$ of a group and in the work of Klopsch and Lev \cite{KL03, KL09} in the context of the ``diameter" of group with respect to the subset $A$. The first systematic and point centric study appeared in the work of Bajnok and Matzke \cite{BM15}, in which, they studied the minimum cardinality of $h$-fold signed sumset $h_{\underline{+}}A$ of subsets of a finite abelian group. In particular, they proved that the minimum cardinality of $h_{\underline{+}}A$ is same as the minimum cardinality of $hA$, when $A$ is a subset of a finite cyclic group. A year later, they \cite{BM2015} classified all possible values of $k$ for which the minimum cardinality of $h_{\underline{+}}A$ coincide with the minimum cardinality of $hA$, when $A$ is a subset of a particular elementary abelian group.

The {\it direct problem} is a problem in which we try to determine the structure and properties of the sumset of a given set. An {\it inverse problem} is a problem in which we attempt to deduce properties of the set $A$ from properties of its sumset. The direct problem for $h\hat{}_{\underline{+}}A$ is to find the minimum number of elements in $h\hat{}_{\underline{+}}A$ in terms of $|A|$. The inverse problem for $h\hat{}_{\underline{+}}A$ is to determine the structure of the finite set $A$ for which $|h\hat{}_{\underline{+}}A|$ is minimal.

Very recently, Bhanja and Pandey \cite{BP} gave some direct and inverse results for the sumset $h_{\underline{+}}A$ in the group of integers. In this article, we study similar direct and inverse problems for restricted signed sumset $h\hat{}_{\underline{+}}A$, when $A$ is a finite set of integers.

\medskip
For any two integers $a$, $b$ $(b\geq a)$, let $[a,b]=\{n\in \mathbb{Z}:a\leq n\leq b\}$. Let $\binom ab=\frac{a!}{b!(a-b)!}$, if $a\geq b$, otherwise 0. We say that $S$ is symmetric, if for all $s\in S$, $-s\in S$.

\section{Direct and inverse theorems for $h\hat{} _{\underline{+}}A$ when $A$ contains only positive integers}\label{sec1}

\begin{theorem}\label{thm2.1}
Let $A$ be a finite set of $k$ positive integers and let $1 \leq h \leq k$ be a positive integer. We have
\begin{equation}\label{eqn2.1}
|h\hat{}_{\underline{+}}A| \geq 2(hk-h^2)+\binom {h+1}2+1.
\end{equation}
The lower bound in (\ref{eqn2.1}) is best possible for $h=1$, $2$ and $k$.
\end{theorem}

\begin{proof}
Let $A=\{a_{0}, a_{1}, \ldots, a_{k-1}\}$, where $0 < a_{0} < a_{1} < \cdots < a_{k-1}$. For $i=0,1,\ldots,k-h-1$ and $j=0,1,\ldots,h$, let
\begin{equation}\label{eqn2.2}
s_{i,j}:=\sum_{\underset{l\neq h-j}{l=0}}^{h} a_{i+l}.
\end{equation}
Let
\begin{equation}\label{eqn2.3}
s_{k-h,0}:=\sum_{l=0}^{h-1} a_{k-h+l}.
\end{equation}
Each $s_{i,j}$ is a sum of $h$ distinct elements of $A$, and hence it is in $h\hat{}_{\underline{+}}A$. Moreover, for $i=0,1,\ldots,k-h-1$ and $j=0,1,\ldots,h-1$, we have
\[
s_{i,j}<s_{i,j+1}~~~\text{and}~~~s_{i,h}=s_{i+1,0}.
\]

Thus, we get at least $hk-h^2+1$ positive integers in $h\hat{}_{\underline{+}}A$. Since $h\hat{}_{\underline{+}}A$ is symmetric, the inverses of these $hk-h^2+1$ integers are also in $h\hat{}_{\underline{+}}A$ with $-s_{0,0}<s_{0,0}$. So, we get $2(hk-h^2+1)$ integers in $h\hat{}_{\underline{+}}A$.

For $i=0,1,\ldots,h-1$ and $j=0,1,\ldots,h-i-1$, define the sequence of integers
\begin{equation}\label{eqn2.4}
t_{i,j}:=\sum_{\underset{l\neq j}{l=0}}^{h-i-1} (-a_{l})+a_{j}+\sum_{m=1}^{i}a_{h-m}.
\end{equation}
Clearly, each $t_{i,j}\in h\hat{}_{\underline{+}}A$. Moreover, for $j=0,1,\ldots,h-i-2$, we have
\[t_{i,j}<t_{i,j+1}, \]
and for $i=0,1,\ldots,h-2$, we have
\[ t_{i,h-i-1}<t_{i+1,0}.\]
Also,
\[
-s_{0,0}<t_{0,0}~~~\text{and}~~~t_{h-1,0}=s_{0,0}.
\]

Therefore, we get $\binom {h+1}2-1$ more integers in $h\hat{}_{\underline{+}}A$ which are listed in (\ref{eqn2.4}). Further, these elements are different from the elements in (\ref{eqn2.2}) and (\ref{eqn2.3}). Hence, we get
\begin{equation*}
|h\hat{}_{\underline{+}}A|\geq 2(hk-h^2)+\binom {h+1}2+1.
\end{equation*}

Next, we show that the lower bound in (\ref{eqn2.1}) is best possible for $h=1$, $2$ and $k$.

Let $h=1$. Then for any finite set $A$ of $k$ positive integers $|1\hat{}_{\underline{+}}A|=2k$ and $2(hk-h^2)+\binom {h+1}2+1=2k$.

Now, let $h=2$ and $A=\{1,3,5,\ldots,2k-1\}$. Then \[2\hat{}_{\underline{+}}A=\{-(4k-4),-(4k-6), \ldots,-2,2,4,\ldots,4k-4\},\]
 and hence $|2\hat{}_{\underline{+}}A|=4k-4=2(hk-h^2)+\binom {h+1}2+1$.

Finally, let $h=k$ and $A=[1,k]$. It is easy to see that $k\hat{}_{\underline{+}}A$ contains either odd integers or even integers. Since, $k\hat{}_{\underline{+}}A\subseteq \left[-\binom {k+1}2,\binom {k+1}2\right]$, we get
\[|k\hat{}_{\underline{+}}A| \leq \binom {k+1}2+1.\]
This together with (\ref{eqn2.1}) give $|k\hat{}_{\underline{+}}A|=\binom {k+1}2+1=2(hk-h^2)+\binom {h+1}2+1$.\\
This completes the proof of the theorem.
\end{proof}

The next two theorems give the inverse results for the cases $h=2$ and $h=k$, respectively. For $h=1$, any set with $k$ elements is extremal.

\begin{theorem}\label{thm2.2}
Let $A$ be a finite set of $k$ $(\geq 2)$ positive integers such that $|2\hat{}_{\underline{+}}A|=4k-4$. Then
$A=\{a_{0},a_{1}\}$ with $a_{0}<a_{1}$, if $k=2$, and $A=d*\{1,3,\ldots,2k-1\}$, for some positive integer $d$, if $k\geq 3$.
\end{theorem}

\begin{proof}
Let $A=\{a_{0}, a_{1}, \ldots, a_{k-1}\}$, where $0<a_{0}<a_{1}<\cdots<a_{k-1}$.
Let
\[|2\hat{}_{\underline{+}}A|=4k-4.\]

First, let $k=2$. Then
\[2\hat{}_{\underline{+}}A=\{a_{0}+a_{1},a_{0}-a_{1},-a_{0}+a_{1},-a_{0}-a_{1}\},\]
where
\[-a_{0}-a_{1}<a_{0}-a_{1}<-a_{0}+a_{1}<a_{0}+a_{1}.\]
Thus, for every set $A$ of two positive integers $|2\hat{}_{\underline{+}}A|=4=4k-4$.

Next, let $k=3$. Then
\begin{align*}
2\hat{}_{\underline{+}}A &= \{a_{0}+a_{1},a_{0}-a_{1},-a_{0}+a_{1},-a_{0}-a_{1},
 a_{0}+a_{2},a_{0}-a_{2},-a_{0}+a_{2},-a_{0}-a_{2},\\
 &\qquad a_{1}+a_{2},a_{1}-a_{2},-a_{1}+a_{2},-a_{1}-a_{2}\}
\end{align*}
where,
\begin{align}\label{eqn2.5}
&-a_{1}-a_{2}<-a_{0}-a_{2}<-a_{0}-a_{1}<a_{0}-a_{1}<-a_{0}+a_{1}<a_{0}+a_{1}\nonumber\\
&\quad <a_{0}+a_{2}<a_{1}+a_{2}.
\end{align}

If $|2\hat{}_{\underline{+}}A|=4k-4=8$, then $2\hat{} _{\underline{+}}A$ contains precisely the integers listed in (\ref{eqn2.5}). Since
\[-a_{0}-a_{2}<a_{0}-a_{2}<a_{0}-a_{1},\]
we get $a_{0}-a_{2}=-a_{0}-a_{1}$, i.e., $a_{2}-a_{1}=2a_{0}$.

Similarly, as
\[
a_{0}-a_{1}<a_{2}-a_{1}<a_{2}-a_{0}=a_{0}+a_{1},
\]
we get $a_{2}-a_{1}=a_{1}-a_{0}$. Hence, $A=a_{0}*\{1,3,5\}$.

Now, let $k=4$. Then
\begin{align*}
2\hat{}_{\underline{+}}A &= \{a_{0}+a_{1},a_{0}-a_{1},-a_{0}+a_{1},-a_{0}-a_{1},
a_{0}+a_{2},a_{0}-a_{2},-a_{0}+a_{2},-a_{0}-a_{2},\\
&\qquad a_{0}+a_{3},a_{0}-a_{3},-a_{0}+a_{3},-a_{0}-a_{3},a_{1}+a_{2},a_{1}-a_{2},-a_{1}+a_{2},-a_{1}-a_{2},\\
&\qquad a_{1}+a_{3},a_{1}-a_{3},-a_{1}+a_{3},-a_{1}-a_{3},a_{2}+a_{3},a_{2}-a_{3},-a_{2}+a_{3},-a_{2}-a_{3}\},
\end{align*}
where
\begin{align}\label{eqn2.6}
&-a_{2}-a_{3}<-a_{1}-a_{3}<-a_{1}-a_{2}<-a_{0}-a_{2}<-a_{0}-a_{1}<a_{0}-a_{1}\nonumber\\
&\quad <-a_{0}+a_{1}<a_{0}+a_{1}<a_{0}+a_{2}<a_{1}+a_{2}<a_{1}+a_{3}<a_{2}+a_{3}.
\end{align}

If $|2\hat{}_{\underline{+}}A|=4k-4=12$, then $2\hat{} _{\underline{+}}A$ contains precisely the integers listed in (\ref{eqn2.6}). Since
\[a_{0}+a_{2} < a_{0}+a_{3} < a_{1}+a_{3},\]
it follows from (\ref{eqn2.6}) that $a_{0}+a_{3}=a_{1}+a_{2}$, which is equivalent to $a_{3}-a_{2}=a_{1}-a_{0}$.

Similarly, since
\[-a_{0}+a_{1} < -a_{0}+a_{2} < a_{0}+a_{2},\]
we have $-a_{0}+a_{2}=a_{0}+a_{1}$, or $a_{2}-a_{1}=2a_{0}$.

We also have
\[-a_{1}-a_{2}=-a_{0}-a_{3} < a_{0}-a_{3} < a_{0}-a_{2}=-a_{0}-a_{1}.\]
Therefore, $a_{0}-a_{3}=-a_{0}-a_{2}$, or $a_{3}-a_{2}=2a_{0}$. Hence, $A=a_{0}*\{1,3,5,7\}$ is the extremal set for all $a_{0}>0$.

Finally, let $k\geq 5$, and $|2\hat{}_{\underline{+}}A|=4k-4$. From Theorem \ref{thm2.1} it follows that the sumset $h\hat{} _{\underline{+}}A$ contains precisely the integers listed in (\ref{eqn2.2}), (\ref{eqn2.3}) and (\ref{eqn2.4}), for $h=2$. Since $2\hat{}A\subseteq [a_{0}+a_{1},a_{k-2}+a_{k-1}]$ and there are exactly $2k-3$ integers in (\ref{eqn2.2}) and (\ref{eqn2.3}) between $a_{0}+a_{1}$ and $a_{k-2}+a_{k-1}$, Theorem \ref{thmB} implies that the set $A$ is in arithmetic progression. That is, the common difference $d=a_{1}-a_{0}=a_{2}-a_{1}=\cdots=a_{k-1}-a_{k-2}$.

Again, since
\[
-a_{0}-a_{2} < -a_{0}-a_{1} < a_{0}-a_{1},
\]
and
\[
-a_{0}-a_{2} < a_{0}-a_{2} < a_{0}-a_{1},
\]
we get $a_{2}-a_{1}=2a_{0}$. Hence $A=a_{0}*\{1,3,\ldots,2k-1\}$.\\
This completes the proof of the theorem.
\end{proof}

\begin{theorem}\label{thm2.3}
Let $A$ be a finite set of $k$ $(\geq 3)$ positive integers such that
\[|k\hat{}_{\underline{+}}A|=\binom {k+1}2+1.\]
Then $A=\{a_{0},a_{1},a_{0}+a_{1}\}$ with $a_{0}<a_{1}$, if $k=3$, and
$A=d*[1,k]$, for some positive integer $d$, if $k\geq 4$.
\end{theorem}

\begin{proof}
First, let $k=3$ and $A=\{a_{0},a_{1},a_{2}\}$, where $0<a_{0}<a_{1}<a_{2}$. Then
\begin{align*}
3\hat{}_{\underline{+}}A &= \{a_{0}+a_{1}+a_{2},a_{0}+a_{1}-a_{2},a_{0}-a_{1}+a_{2},a_{0}-a_{1}-a_{2},
-a_{0}+a_{1}+a_{2},\\
&\qquad -a_{0}+a_{1}-a_{2},-a_{0}-a_{1}+a_{2},-a_{0}-a_{1}-a_{2}\},
\end{align*}
where, we have
\begin{align}\label{eqn2.7}
&-a_{0}-a_{1}-a_{2}<a_{0}-a_{1}-a_{2}<-a_{0}+a_{1}-a_{2}<-a_{0}-a_{1}+a_{2}\nonumber\\
&\quad <a_{0}-a_{1}+a_{2}<-a_{0}+a_{1}+a_{2}<a_{0}+a_{1}+a_{2}.
\end{align}
If $|3\hat{}_{\underline{+}}A|=\binom 42+1=7$, then $3\hat{}_{\underline{+}}A$ contains precisely the seven integers in (\ref{eqn2.7}). Since
\[-a_{0}+a_{1}-a_{2}<a_{0}+a_{1}-a_{2}<a_{0}-a_{1}+a_{2},\]
we have $a_{0}+a_{1}-a_{2}=-a_{0}-a_{1}+a_{2}$, i.e., $a_{2}-a_{1}=a_{0}$. Hence, $A=\{a_{0},a_{1},a_{0}+a_{1}\}$ is an extremal set.

Next, let $k=4$ and $A=\{a_{0},a_{1},a_{2},a_{3}\}$, where $0<a_{0}<a_{1}<a_{2}<a_{3}$. Let $|4\hat{} _{\underline{+}}A|=\binom 52+1=11$. Then $4\hat{}_{\underline{+}}A$ contains precisely the following sequence of integers written in an increasing order.
\begin{align}\label{eqn2.8}
&-a_{0}-a_{1}-a_{2}-a_{3}<a_{0}-a_{1}-a_{2}-a_{3}<-a_{0}+a_{1}-a_{2}-a_{3}\nonumber\\
&\quad <-a_{0}-a_{1}+a_{2}-a_{3}<-a_{0}-a_{1}-a_{2}+a_{3}<a_{0}-a_{1}-a_{2}+a_{3}\nonumber\\
&\quad <-a_{0}+a_{1}-a_{2}+a_{3}<-a_{0}-a_{1}+a_{2}+a_{3}<a_{0}-a_{1}+a_{2}+a_{3}\nonumber\\
&\quad <-a_{0}+a_{1}+a_{2}+a_{3}<a_{0}+a_{1}+a_{2}+a_{3}.
\end{align}

Since the sumset $4\hat{}_{\underline{+}}A$ is symmetric, from (\ref{eqn2.8}) it follows that
\[
-a_{0}-a_{1}+a_{2}-a_{3}=-(-a_{0}-a_{1}+a_{2}+a_{3}),
\]
\[
-a_{0}-a_{1}-a_{2}+a_{3}=-(-a_{0}+a_{1}-a_{2}+a_{3}),
\]
and
\[
a_{0}-a_{1}-a_{2}+a_{3}=0.
\]
These above three equations give $a_{3}-a_{2}=a_{2}-a_{1}=a_{1}-a_{0}=a_{0}$. Hence, $A=a_{0}*\{1,2,3,4\}$.

Finally, let $k\geq 5$ and $A=\{a_{0}, a_{1}, \ldots, a_{k-1}\}$, where $0<a_{0}<a_{1}<\cdots<a_{k-1}$. Let
\begin{equation}\label{eqn2.9}
|k\hat{} _{\underline{+}}A|=\binom {k+1}2+1.
\end{equation}
Then, $k\hat{}_{\underline{+}}A$ contains precisely the integers listed in (\ref{eqn2.4}), with one more integer $-a_{0}-a_{1}-\cdots-a_{k-1}$. For $j=1,2,\ldots,k-1$, set
\begin{align}\label{eqn2.10}
u_{j}=a_{0}+\sum_{\underset{l\neq j}{l=1}}^{k-1} (-a_{l})+a_{j}.
\end{align}
Clearly,
\[
t_{0,1}<u_{1}<u_{2}<\cdots<u_{k-2}<u_{k-1}=t_{1,0}.
\]

So, there are exactly $k-2$ distinct integers in (\ref{eqn2.10}) between $t_{0,1}$ and $t_{1,0}$. Therefore, (\ref{eqn2.4}), (\ref{eqn2.9}) and (\ref{eqn2.10}) imply that, for $j=1,2,\ldots,k-2$,
\[t_{0,j+1}=u_{j}.\]
This is equivalent to $a_{j+1}-a_{j}=a_{0}$, for $j=1,2,\ldots,k-2$. That is
\[
a_{k-1}-a_{k-2}=\cdots=a_{3}-a_{2}=a_{2}-a_{1}=a_{0}.
\]

Again, since $k\hat{}_{\underline{+}}A$ is symmetric, we have $-t_{0,0}=t_{k-3,0}$, i.e.,
\[
-(-a_{0}-a_{1}-a_{2}-a_{3}+a_{4}-\cdots-a_{k-1})=a_{0}-a_{1}-a_{2}+a_{3}+a_{4}+\cdots+a_{k-1},
\]
or
\[
a_{4}=a_{1}+a_{2}.
\]
Since $a_{3}-a_{2}=a_{0}$, we get $a_{4}-a_{3}=a_{1}-a_{0}$. Hence, $A=a_{0}*[1,k]$.\\
This completes the proof of the theorem.
\end{proof}

For $h\geq 3$, we believe that the sumset $h\hat{}_{\underline{+}}A$ contains at least $2hk-h^2+1$ integers. So we conjecture that
\begin{conjecture}\label{conj2.1}
Let $A$ be a finite set of $k$ $(\geq 4)$ positive integers and let $3 \leq h \leq k-1$. Then
\begin{equation}\label{eqn2.11}
|h\hat{}_{\underline{+}}A| \geq 2hk-h^2+1.
\end{equation}
The lower bound in (\ref{eqn2.11}) is best possible.
\end{conjecture}

The following example confirms the conjecture in a very special case. Also in Theorem \ref{thm2.4}, we prove the conjecture for $h=3$. Moreover, we also give the inverse result in this case.

\noindent {\bf Example 1} {\bf (Super increasing sequence).}
Let $A=\{a_{0},a_{1},\ldots,a_{k-1}\}$, where $k\geq 6$, $a_{0}>0$, and $a_{i}>\sum_{j=0}^{i-1}a_{j}$, for $i=1,2,\ldots,k-1$.

Let $h\geq 5$. Clearly, the sumset $h\hat{}_{\underline{+}}A$ contains at least $2(hk-h^2)+\binom {h+1}2+1$ integers, which are listed in (\ref{eqn2.2}), (\ref{eqn2.3}) and (\ref{eqn2.4}).

For $j=1,2,\ldots,h-2$, consider the integers $-2a_{0}+s_{0,j}$. Clearly
\[
-2a_{0}+s_{0,j}=-a_{0}+\sum_{\underset{l\neq h-j}{l=1}}^{h} a_{l}\in h\hat{}_{\underline{+}}A,
\]
and
\[
s_{0,j-1}<-2a_{0}+s_{0,j}<s_{0,j}.
\]
So, we get $h-2$ extra positive integers $h\hat{}_{\underline{+}}A$, which are not present in (\ref{eqn2.2}), (\ref{eqn2.3}) and (\ref{eqn2.4}). Since
\[
-s_{0,j}<-(-2a_{0}+s_{0,j})<-s_{0,j-1},
\]
we get $h-2$ further extra integers in $h\hat{}_{\underline{+}}A$.

Also, for $j=2,3,\ldots,h-3$, consider the integers
\begin{equation}\label{eqn2.12}
t_{0,h-j-1}<-t_{j,h-j-2}<-t_{j,h-j-3}<\cdots<-t_{j,0}<-t_{j-1,h-j}<t_{0,h-j}.
\end{equation}

Then, for $j=2,3,\ldots,h-3$, we get $h-j$ extra integers. Therefore, we get $3+4+\cdots+(h-2)=\binom{h}{2}-h-2$ more integers in $h\hat{}_{\underline{+}}A$ which are listed in (\ref{eqn2.12}) and never counted before. We also get one more integer, i.e., $-t_{h-3,2}$ such that $t_{0,1}<-t_{h-3,2}<t_{0,2}$. So, we get $2(h-2)+\binom{h}{2}-h-2+1=\binom{h}{2}+(h-5)$ extra integers. Hence, by and large, we have
\[
|h\hat{}_{\underline{+}}A| \geq 2hk-h^2+h-4\geq 2hk-h^2+1.
\]

\begin{theorem}\label{thm2.4}
Let $A$ be a finite set of $k$ $(\geq 4)$ positive integers. Then
\begin{equation}\label{eqn2.13}
|3\hat{}_{\pm}A| \geq 6k-8.
\end{equation}
Moreover, if $|3\hat{}_{\pm}A|=6k-8$, then $A=d*\{1,3,5,\ldots,2k-1\}$, for some positive integer $d$.
\end{theorem}

\begin{proof}
Let $A=\{a_{0},a_{1},\ldots,a_{k-1}\}$, where $0<a_{0}<a_{1}<\cdots<a_{k-1}$. From Theorem \ref{thm2.1}, we have $|3\hat{}_{\pm}A|\geq 6k-11$.

Next, we show that there exist at least three extra integers in $3\hat{}_{\pm}A$ which are not counted in Theorem \ref{thm2.1}. Consider the following thirteen integers of $3\hat{}_{\pm}A$:
\begin{align}\label{eqn2.14}
&-a_{1}-a_{2}-a_{3}<-a_{0}-a_{2}-a_{3}<-a_{0}-a_{1}-a_{3}<-a_{0}-a_{1}-a_{2}\nonumber\\
&\quad <a_{0}-a_{1}-a_{2}<-a_{0}+a_{1}-a_{2} <-a_{0}-a_{1}+a_{2}<a_{0}-a_{1}+a_{2}\nonumber\\
&\quad <-a_{0}+a_{1}+a_{2}<a_{0}+a_{1}+a_{2}<a_{0}+a_{1}+a_{3}<a_{0}+a_{2}+a_{3}\nonumber\\
&\quad <a_{1}+a_{2}+a_{3}.
\end{align}

We exhibit at least three extra integers between $-a_{1}-a_{2}-a_{3}$ and $a_{1}+a_{2}+a_{3}$ in all possible cases.

{\bf Case 1:}
Let $a_{3}-a_{2}<a_{3}-a_{1}<2a_{0}$. Then, we get at least two extra positive integers $-a_{0}+a_{1}+a_{3}$ and $-a_{0}+a_{2}+a_{3}$ which are not present in (\ref{eqn2.14}) such that
\begin{equation*}
-a_{0}+a_{1}+a_{2}<-a_{0}+a_{1}+a_{3}<-a_{0}+a_{2}+a_{3}<a_{0}+a_{1}+a_{2}.
\end{equation*}

{\bf Case 2:}
Let $a_{3}-a_{2}<2a_{0}<a_{3}-a_{1}$. Then, we get at least two extra positive integers $-a_{0}-a_{1}+a_{3}$ and $-a_{0}+a_{1}+a_{3}$ which are not present in (\ref{eqn2.14}) such that
\begin{align*}
&-a_{0}-a_{1}+a_{2}<-a_{0}-a_{1}+a_{3}<a_{0}-a_{1}+a_{2}<-a_{0}+a_{1}+a_{2}<-a_{0}+a_{1}+a_{3}\\
&\quad <a_{0}+a_{1}+a_{2}.
\end{align*}

{\bf Case 3:}
Let $a_{3}-a_{1}>a_{3}-a_{2}>2a_{0}$. Then, we get an extra positive integer $-a_{0}+a_{1}+a_{3}$ such that
\begin{equation*}
a_{0}+a_{1}+a_{2}<-a_{0}+a_{1}+a_{3}<a_{0}+a_{1}+a_{3}.
\end{equation*}

To exhibit one further extra positive integer consider the following subcases

{\bf Subcase (i)}
($a_{2}-a_{1}<2a_{0}$). We get one more extra positive integer $-a_{0}+a_{2}+a_{3}$ such that
\begin{equation*}
a_{0}+a_{1}+a_{2}<-a_{0}+a_{1}+a_{3}<-a_{0}+a_{2}+a_{3}<a_{0}+a_{1}+a_{3}.
\end{equation*}

{\bf Subcase (ii)}
($a_{2}-a_{1}>2a_{0}$). We get one more extra positive integer $-a_{0}+a_{2}+a_{3}$ such that
\begin{equation*}
a_{0}+a_{1}+a_{2}<-a_{0}+a_{1}+a_{3}<a_{0}+a_{1}+a_{3}<-a_{0}+a_{2}+a_{3}<a_{0}+a_{2}+a_{3}.
\end{equation*}

{\bf Subcase (iii)}
($a_{2}-a_{1}=2a_{0}$). In this subcase, we get two positive integers $a_{0}-a_{1}+a_{3}$ and $a_{0}-a_{2}+a_{3}$ such that
\begin{equation*}
a_{0}-a_{1}+a_{2}=3a_{0}<a_{0}-a_{2}+a_{3}<a_{0}-a_{1}+a_{3}<-a_{0}+a_{1}+a_{3}<a_{0}+a_{1}+a_{3}.
\end{equation*}

But, we already have
\begin{equation*}
a_{0}-a_{1}+a_{2}<-a_{0}+a_{1}+a_{2}<a_{0}+a_{1}+a_{2}<-a_{0}+a_{1}+a_{3}<a_{0}+a_{1}+a_{3}.
\end{equation*}

Thus, except in the cases $a_{0}-a_{2}+a_{3}=-a_{0}+a_{1}+a_{2}$ and $a_{0}-a_{1}+a_{3}=a_{0}+a_{1}+a_{2}$, we get at least one extra positive integer and hence we are done.

So, let
\[
a_{0}-a_{2}+a_{3}=-a_{0}+a_{1}+a_{2},
\]
and
\[a_{0}-a_{1}+a_{3}=a_{0}+a_{1}+a_{2}.
\]
These two equations imply
\[
2(a_{2}-a_{0})=a_{3}-a_{1}=a_{1}+a_{2}.
\]
Consider the integer $-a_{0}-a_{2}+a_{3}$. We have
\begin{equation*}
-a_{0}-a_{1}+a_{2}=a_{0}<-a_{0}-a_{2}+a_{3}<-a_{0}-a_{1}+a_{3}=-a_{0}+a_{1}+a_{2}.
\end{equation*}

If $-a_{0}-a_{2}+a_{3}\neq a_{0}-a_{1}+a_{2}$, then we are done, as we get one extra positive integer. Otherwise, let
\begin{equation*}
-a_{0}-a_{2}+a_{3}=a_{0}-a_{1}+a_{2},
\end{equation*}
or
\[
a_{3}-a_{2}=2a_{0}-a_{1}+a_{2}=4a_{0}.
\]
Therefore, we have
\[
a_{3}-a_{1}=a_{3}-a_{2}+a_{2}-a_{1}=6a_{0},
\]
and
\[
a_{2}-a_{0}=\frac{1}{2}(a_{3}-a_{1})=3a_{0}.
\]

Solving these equations we get $a_{1}=2a_{0}$, $a_{2}=4a_{0}$ and $a_{3}=8a_{0}$. Thus, we get one extra positive integer $-a_{1}+a_{2}+a_{3}$ such that
\begin{equation*}
-a_{0}+a_{1}+a_{3}=9a_{0}<10a_{0}=-a_{1}+a_{2}+a_{3}<11a_{0}=a_{0}+a_{1}+a_{3}.
\end{equation*}
Hence, we get at least two extra positive integers in every case.

{\bf Case 4:}
Let $a_{3}-a_{2}<a_{3}-a_{1}=2a_{0}$. Then we get at least two extra positive integers $-a_{0}-a_{1}+a_{3}$ and $-a_{0}+a_{1}+a_{3}$ which are not present in (\ref{eqn2.14}) such that
\begin{align*}
&-a_{0}-a_{1}+a_{2}<-a_{0}-a_{1}+a_{3}=a_{0}<a_{0}-a_{1}+a_{2}<-a_{0}+a_{1}+a_{2}\\
&\quad <-a_{0}+a_{1}+a_{3}<a_{0}+a_{1}+a_{2}.
\end{align*}

{\bf Case 5:}
Let $a_{3}-a_{1}>a_{3}-a_{2}=2a_{0}$. We consider the following three subcases:

{\bf Subcase (i)}
Let $a_{2}-a_{1}<2a_{0}$. Then, we get at least two extra positive integers $-a_{0}-a_{2}+a_{3}$ and $-a_{0}+a_{2}+a_{3}$ such that
\begin{align*}
&-a_{0}-a_{1}+a_{2}<a_{0}=-a_{0}-a_{2}+a_{3}<a_{0}-a_{1}+a_{2}<-a_{0}+a_{1}+a_{2}\\
&\quad <a_{0}+a_{1}+a_{2}<-a_{0}+a_{2}+a_{3}<a_{0}+a_{1}+a_{3}.
\end{align*}

{\bf Subcase (ii)}
Let $a_{2}-a_{1}>2a_{0}$. Then, we get two extra positive integers $-a_{0}-a_{2}+a_{3}$ and $-a_{0}+a_{2}+a_{3}$ such that
\begin{align*}
&a_{0}+a_{1}-a_{2}<-a_{0}<a_{0}=-a_{0}-a_{2}+a_{3}<-a_{0}-a_{1}+a_{2}<a_{0}-a_{1}+a_{2}\\
&\quad <-a_{0}+a_{1}+a_{2}<a_{0}+a_{1}+a_{2}<a_{0}+a_{1}+a_{3}<-a_{0}+a_{2}+a_{3}<a_{0}+a_{2}+a_{3}.
\end{align*}

{\bf Subcase (iii)}
Let $a_{2}-a_{1}=2a_{0}$. We get an extra positive integer $a_{1}-a_{2}+a_{3}$ such that
\begin{align*}
a_{0}-a_{1}+a_{2}=3a_{0}<2a_{0}+a_{1}=a_{1}-a_{2}+a_{3}<a_{0}+2a_{1}=-a_{0}+a_{1}+a_{2}.
\end{align*}

If $a_{1}-a_{0}>2a_{0}$, then we get one more extra positive integer $a_{0}-a_{1}+a_{3}$ such that
\begin{align*}
a_{0}-a_{1}+a_{2}<a_{0}-a_{1}+a_{3}<-2a_{0}+a_{3}=a_{1}-a_{2}+a_{3}<-a_{0}+a_{1}+a_{2}.
\end{align*}

If $a_{1}-a_{0}<2a_{0}$, then we get one more extra positive integer $-a_{1}+a_{2}+a_{3}$ such that
\begin{align*}
&a_{0}-a_{1}+a_{2}<a_{1}-a_{2}+a_{3}<-a_{0}+a_{1}+a_{2}<a_{0}+a_{1}+a_{2}<-a_{1}+a_{2}+a_{3}\\
&\quad <a_{0}+a_{1}+a_{3}.
\end{align*}

Let $a_{1}-a_{0}=2a_{0}$. Then, the integer $-a_{0}-a_{1}+a_{2}=a_{0}$ is positive. So, the inverse of this integer gives one more extra integer with
\begin{align*}
&-a_{0}+a_{1}-a_{2}<a_{0}+a_{1}-a_{2}<-a_{0}-a_{1}+a_{2}<a_{0}-a_{1}+a_{2}.
\end{align*}

From the above discussion, we conclude that except in the case $a_{1}-a_{0}=a_{2}-a_{1}=a_{3}-a_{2}=2a_{0}$, we get at least two extra positive integers in $3\hat{}_{\pm}A$, which are not present in (\ref{eqn2.14}). Since, the inverses of these integers are negative, we get two more extra integers. So, total we get at least four extra integers in $3\hat{}_{\pm}A$, which are not included in (\ref{eqn2.14}). In case $a_{1}-a_{0}=a_{2}-a_{1}=a_{3}-a_{2}=2a_{0}$, we get at least three extra integers. Therefore, in each case we get at least three extra integers in $3\hat{}_{\pm}A$, which are not present in (\ref{eqn2.14}). Hence, $|3\hat{}_{\pm}A|\geq 6k-8$. This establishes (\ref{eqn2.13}).

Moreover, if $|3\hat{}_{\pm}A|=6k-8$, then $a_{1}-a_{0}=a_{2}-a_{1}=a_{3}-a_{2}=2a_{0}$.

Now, let $|3\hat{}_{\pm}A|=6k-8$. If $k=4$, then we are done, as $A=\{a_{0},3a_{0},5a_{0},7a_{0}\}$ $=a_{0}*\{1,3,5,7\}$.

Let $k\geq 5$, and let $A'=A\setminus \{a_{0}\}$. Therefore, $A'$ is a finite set of $k-1$ positive integers such that $3\hat{}A'\subseteq [a_{1}+a_{2}+a_{3},a_{k-3}+a_{k-2}+a_{k-1}]$. Since $|3\hat{}_{\pm}A|=6k-8$, from the above proof it follows that $|3\hat{}A'|=3k-11$. Thus, Theorem \ref{thmB} implies that the set $A'$ is in arithmetic progression, i.e.,
\[
a_{k-1}-a_{k-2}=a_{k-2}-a_{k-3}=\cdots=a_{2}-a_{1}=d.
\]
Hence
\[
A=a_{0}*\{1,3,5,\ldots,2k-1\}.
\]
This completes the proof of the theorem.
\end{proof}

Now, we conjecture the inverse result as follows:
\begin{conjecture}\label{conj2.2}
Let $A$ be a finite set of $k$ $(\geq 4)$ positive integers and let $3 \leq h \leq k-1$. If $|h\hat{} _{\underline{+}}A|=2hk-h^2+1$, then $A=d*\{1,3,\ldots,2k-1\}$, for some positive integer $d$.
\end{conjecture}

Theorem \ref{thm2.4} confirms Conjecture \ref{conj2.2} for $h=3$.

\section{Direct and inverse theorems for $h\hat{}_{\underline{+}}A$ when $A$ contains non-negative integers with $0\in A$}\label{sec2}

\begin{theorem}\label{thm3.1}
Let $A$ be a finite set of $k$ non-negative integers with $0 \in A$. Let $1 \leq h \leq k$ be a positive integer. Then
\begin{equation}\label{eqn3.1}
|h\hat{}_{\underline{+}}A| \geq 2(hk-h^2)+\binom h2+1.
\end{equation}
The lower bound in (\ref{eqn3.1}) is best possible for $h=1$, $2$ and $k$.
\end{theorem}

\begin{proof}
Let $A=\{a_{0}, a_{1}, \ldots, a_{k-1}\}$, where $0=a_{0} < a_{1} < \cdots < a_{k-1}$.
 From (\ref{eqn2.2}) and (\ref{eqn2.3}), it follows that $h\hat{}_{\underline{+}}A$ contains at least $hk-h^2+1$ positive integers and hence including their inverses, $h\hat{}_{\underline{+}}A$ contains at least $2(hk-h^2+1)$ integers.

Again, since $a_{0}=0$, from (\ref{eqn2.4}) it follows that, for $i=0,1,\ldots,h-2$, we have $t_{i,h-i-1}=t_{i+1,0}$, $-s_{0,0}=t_{0,0}$ and $t_{h-1,0}=s_{0,0}$. Thus, we get $\binom h2-1$ extra integers in $h\hat{}_{\underline{+}}A$ from the list (\ref{eqn2.4}). Hence
\[
|h\hat{}_{\underline{+}}A| \geq 2(hk-h^2)+\binom h2+1.
\]

Next, we show that the lower bound in (\ref{eqn3.1}) is best possible for $h=1$, $2$ and $k$.

If $h=1$, then for any finite set $A$ of $k$ non-negative integers with $0\in A$, we have $|1\hat{}_{\underline{+}}A|=2k-1$ and $2(hk-h^2)+\binom h2+1=2k-1$.

Now, let $h=2$ and $A=[0,k-1]$. Then
\[
2\hat{}_{\underline{+}}A=[-(2k-3),(2k-3)]\setminus\{0\}.
\]
So, $|2\hat{}_{\underline{+}}A|=4k-6=2(hk-h^2)+\binom h2+1$.

Finally, let $h=k$ and $A=[0,k-1]$. Then, it is easy to see that $k\hat{}_{\underline{+}}A$ contains either odd integers or even integers. Since $k\hat{}_{\underline{+}}A\subseteq \left[-\binom k2,\binom k2\right]$, we get \[|k\hat{}_{\underline{+}}A|\leq \binom k2+1.\]
This together with (\ref{eqn3.1}) give $|k\hat{}_{\underline{+}}A|=\binom k2+1=2(hk-h^2)+\binom h2+1$. \\
This completes the proof of the theorem.
\end{proof}

We now give inverse results for $h=2$ and $h=k$ in theorems \ref{thm3.2} and \ref{thm3.3} respectively.

\begin{theorem}\label{thm3.2}
Let $A$ be a finite set of $k$ $(\geq 2)$ non-negative integers with $0\in A$ such that $|2\hat{}_{\underline{+}}A|=4k-6$. Then
\[A=\begin{cases}
\{0,a\},~a>0,~~\text{if}~k=2;\\
 d*[0,k-1],~\text{for some positive integer}~d,~~\text{if}~k\geq 3.
\end{cases}\]
\end{theorem}

\begin{proof}
Let $A=\{0, a_{1},a_{2}, \ldots, a_{k-1}\}$, where $0<a_{1}<a_{2}<\cdots<a_{k-1}$. Let
\[|2\hat{}_{\underline{+}}A|=4k-6.\]

First, let $k=2$. Then $2\hat{}_{\underline{+}}A=\{a_{1},-a_{1}\}$. So, $|2\hat{}_{\underline{+}}A|=2=\binom {2}{2}+1$. Thus, $A$ is an extremal set.

Next, let $k=3$. Then
\begin{align*}
2\hat{}_{\underline{+}}A &=\{a_{1},-a_{1},a_{2},-a_{2},a_{1}+a_{2},a_{1}-a_{2},-a_{1}+a_{2},-a_{1}-a_{2}\},
\end{align*}
where
\begin{equation}\label{eqn3.2}
-a_{1}-a_{2}<-a_{2}<-a_{1}<a_{1}<a_{2}<a_{1}+a_{2}.
\end{equation}
If $|2\hat{}_{\underline{+}}A|=6=4k-6$, then $2\hat{}_{\underline{+}}A$ contains precisely the integers listed in (\ref{eqn3.2}). Since
\[
-a_{2}<a_{1}-a_{2}<a_{1},
\]
from (\ref{eqn3.2}) it follows that $a_{1}-a_{2}=-a_{1}$, i.e., $a_{2}-a_{1}=a_{1}$. Hence, $A=\{0,a_{1},2a_{1}\}$.

Now, let $k=4$. Then
\begin{align*}
2\hat{}_{\underline{+}}A &= \{a_{1},-a_{1},a_{2},-a_{2},a_{3},-a_{3},a_{1}+a_{2},
a_{1}-a_{2},-a_{1}+a_{2},-a_{1}-a_{2},a_{1}+a_{3},\\
&\qquad a_{1}-a_{3},-a_{1}+a_{3},-a_{1}-a_{3},a_{2}+a_{3},a_{2}-a_{3},-a_{2}+a_{3},-a_{2}-a_{3}\},
\end{align*}
where
\begin{align}\label{eqn3.3}
&-a_{2}-a_{3}<-a_{1}-a_{3}<-a_{1}-a_{2}<-a_{2}<-a_{1}<a_{1}<a_{2}<a_{1}+a_{2}\nonumber\\
&\quad <a_{1}+a_{3}<a_{2}+a_{3}.
\end{align}
If $|2\hat{}_{\underline{+}}A|=10=4k-6$, then $2\hat{}_{\underline{+}}A$ contains precisely the integers listed in (\ref{eqn3.3}). Since
\[
a_{2}<a_{3}<a_{1}+a_{3},
\]
from (\ref{eqn3.3}) it follows that $a_{3}=a_{1}+a_{2}$, or $a_{3}-a_{2}=a_{1}$.

Similarly,
\[
-a_{2}<a_{1}-a_{2}<a_{1}
\]
imply $a_{1}-a_{2}=-a_{1}$, or $a_{2}-a_{1}=a_{1}$.
Hence, $A=\{0,a_{1},2a_{1},3a_{1}\}$.

Finally, let $k\geq 5$, and $|2\hat{}_{\underline{+}}A|=4k-6$. From Theorem \ref{thmA} we know that $|2\hat{}A|\geq 2k-3$, and since $2\hat{}A \cap (-2\hat{}A)=\emptyset$, we get $|2\hat{}A|=2k-3$. Therefore, by Theorem \ref{thmB}, the set $A$ is in arithmetic progression with the common difference $a_{k-1}-a_{k-2}=a_{k-2}-a_{k-3}=\cdots=a_{1}-a_{0}=a_{1}$. Hence, $A=a_{1}*[0,k-1]$.
\end{proof}

\begin{theorem}\label{thm3.3}
Let $A$ be a finite set of $k$ $(\geq 3)$ non-negative integers with $0\in A$ such that
$|k\hat{}_{\underline{+}}A|=\binom k2+1$. Then
\[A=\begin{cases}
\{0,a_{1},a_{2}\},~0<a_{1}<a_{2},~~\text{if}~k=3;\\
\{0,a_{1},a_{2},a_{1}+a_{2}\},~0<a_{1}<a_{2},~~\text{if}~k=4;\\
 d*[0,k-1],~\text{for some positive integer}~d,~~\text{if}~k\geq 5.
\end{cases}\]
\end{theorem}

\begin{proof}
Let $A=\{0, a_{1},a_{2}, \ldots, a_{k-1}\}$, where $0<a_{1}<a_{2}<\cdots<a_{k-1}$. Let
\[|k\hat{}_{\underline{+}}A|=\binom k2+1.\]

First, let $k=3$. Then
\[3\hat{}_{\underline{+}}A=\{a_{1}+a_{2},a_{1}-a_{2},-a_{1}+a_{2},-a_{1}-a_{2}\},\]
where
\[-a_{1}-a_{2}<a_{1}-a_{2}<-a_{1}+a_{2}<a_{1}+a_{2}.\]
So, $|3\hat{}_{\underline{+}}A|=4=\binom {3}{2}+1$. Thus, $A$ is an extremal set.

Next, let $k=4$. Then
\begin{align*}
4\hat{}_{\underline{+}}A &= \{a_{1}+a_{2}+a_{3},a_{1}+a_{2}-a_{3},a_{1}-a_{2}+a_{3},a_{1}-a_{2}-a_{3},
-a_{1}+a_{2}+a_{3},\\
&\qquad -a_{1}+a_{2}-a_{3},-a_{1}-a_{2}+a_{3},-a_{1}-a_{2}-a_{3}\},
\end{align*}
where
\begin{align}\label{eqn3.4}
&-a_{1}-a_{2}-a_{3}<a_{1}-a_{2}-a_{3}<-a_{1}+a_{2}-a_{3}<-a_{1}-a_{2}+a_{3}\nonumber\\
&\quad <a_{1}-a_{2}+a_{3}<-a_{1}+a_{2}+a_{3}<a_{1}+a_{2}+a_{3}.
\end{align}
If $|4\hat{}_{\underline{+}}A|=\binom 42+1=7$, then $4\hat{}_{\underline{+}}A$ contains precisely the above  seven integers in (\ref{eqn3.4}). Since
\[
-a_{1}+a_{2}-a_{3}<a_{1}+a_{2}-a_{3}<a_{1}-a_{2}+a_{3},
\]
we have $a_{1}+a_{2}-a_{3}=-a_{1}-a_{2}+a_{3}$, i.e., $a_{3}-a_{2}=a_{1}$. Hence, $A=\{0,a_{1},a_{2},a_{1}+a_{2}\}$ is an extremal set.

Finally, let $k\geq 5$, and $|k\hat{}_{\underline{+}}A|=\binom k2+1$. Let $A'=A\setminus\{0\}$. So, $A'$ is a finite set of $k-1$ positive integers with $k\hat{}_{\underline{+}}A=(k-1)\hat{}_{\underline{+}}A'$. Since $|(k-1)\hat{}_{\underline{+}}A'|=|k\hat{}_{\underline{+}}A|=\binom k2+1$, Theorem \ref{thm2.3} implies that the set $A'$ is in arithmetic progression with the common difference $a_{1}$, the smallest element in $A'$. Hence  $A=a_{1}*\{0,1,2,\ldots,k-1\}=a_{1}*[0,k-1]$.
\end{proof}

For $h\geq 3$, we believe that the sumset $h\hat{}_{\underline{+}}A$ contains at least $2hk-h(h+1)+1$ integers. So, we conjecture that

\begin{conjecture}\label{conj3.1}
Let $A$ be a finite set of $k$ $(\geq 5)$ non-negative integers with $0 \in A$. Let $3\leq h\leq k-1$ be a positive integer. Then
\begin{equation}\label{eqn3.5}
|h\hat{}_{\underline{+}}A| \geq 2hk-h(h+1)+1.
\end{equation}
The lower bound in (\ref{eqn3.5}) is best possible.
\end{conjecture}

We confirm Conjecture \ref{conj3.1} for $h=3$. Moreover, we also give the inverse result in this case.

\begin{theorem}\label{thm3.4}
Let $A$ be a finite set of $k$ $(\geq 5)$ non-negative integers with $0\in A$. Then
\begin{equation}\label{eqn3.6}
|3\hat{}_{\pm}A| \geq 6k-11.
\end{equation}
Moreover, if $|3\hat{}_{\pm}A|=6k-11$, then $A=d*[0,k-1]$.
\end{theorem}

\begin{proof}
Let $A=\{0,a_{1},a_{2},\ldots,a_{k-1}\}$, where $0<a_{1}<a_{2}<\cdots<a_{k-1}$. From Theorem \ref{thm3.1}, it follows that $|3\hat{}_{\pm}A|\geq 6k-14$.

Next, we show that there exists at least three extra integers in $3\hat{}_{\pm}A$ which are not counted in Theorem \ref{thm3.1}. Consider the following twelve integers of $3\hat{}_{\pm}A$:
\begin{align}\label{eqn3.7}
&-a_{1}-a_{2}-a_{4}<-a_{1}-a_{2}-a_{3}<-a_{2}-a_{3}<-a_{1}-a_{3}<-a_{1}-a_{2}\nonumber\\
&\quad <a_{1}-a_{2}<-a_{1}+a_{2}<a_{1}+a_{2}<a_{1}+a_{3}<a_{2}+a_{3}<a_{1}+a_{2}+a_{3}\nonumber\\
&\quad <a_{1}+a_{2}+a_{4}.
\end{align}

We exhibit at least three extra integers in between $-a_{1}-a_{2}-a_{4}$ and $a_{1}+a_{2}+a_{4}$ in all possible cases.

{\bf Case 1:}
Let $a_{3}-a_{2}<a_{1}$. Then, we have
\[
a_{1}-a_{2}<-a_{2}+a_{3}<-a_{1}+a_{3}<a_{1}+a_{2},
\]
and
\[
a_{1}-a_{2}<-a_{1}+a_{2}<-a_{1}+a_{3}.
\]

If $-a_{2}+a_{3}\neq -a_{1}+a_{2}$, then we get two extra positive integers $-a_{2}+a_{3}$ and $-a_{1}+a_{3}$.

So, let $-a_{2}+a_{3}=-a_{1}+a_{2}$. If $a_{3}-a_{1}<a_{1}$, then we get two extra positive integers $-a_{1}+a_{3}$ and $-a_{1}+a_{2}+a_{3}$ such that
\[
-a_{1}+a_{2}<-a_{1}+a_{3}<-a_{1}+a_{2}+a_{3}<a_{1}+a_{2}.
\]

If $a_{3}-a_{1}>a_{1}$, then we get two extra positive integers $-a_{1}+a_{3}$ and $-a_{1}+a_{2}+a_{3}$ such that
\[
-a_{1}+a_{2}<-a_{1}+a_{3}<a_{1}+a_{2}<-a_{1}+a_{2}+a_{3}<a_{1}+a_{3}.
\]

If $a_{3}-a_{1}=a_{1}$, then also we get two extra positive integers $-a_{1}+a_{3}$ and $a_{1}-a_{2}+a_{3}$ such that
\[
-a_{1}+a_{2}<-a_{1}+a_{3}<a_{1}-a_{2}+a_{3}<a_{1}+a_{2}.
\]

{\bf Case 2:}
Let $a_{3}-a_{2}=a_{1}$. Then, by similar arguments to Case 1, unless $-a_{2}+a_{3}=-a_{1}+a_{2}$, we get two extra positive integers $-a_{2}+a_{3}$ and $-a_{1}+a_{3}$.

Let $-a_{2}+a_{3}=-a_{1}+a_{2}$. Then we get an extra positive integer $-a_{1}+a_{3}$ such that
\[
-a_{1}+a_{2}<-a_{1}+a_{3}<a_{1}+a_{2}.
\]
Again, we get one more extra integer $-a_{1}-a_{2}+a_{3}=0$ such that
\[
a_{1}-a_{2}<-a_{1}-a_{2}+a_{3}<-a_{1}+a_{2}.
\]

{\bf Case 3:}
Let $a_{3}-a_{2}>a_{1}$. So, $a_{3}-a_{1}>a_{1}$.

{\bf Subcase (i).}
Let $-a_{1}+a_{3}<a_{1}+a_{2}$. Unless $-a_{2}+a_{3}=-a_{1}+a_{2}$, we get two extra positive integers $-a_{2}+a_{3}$ and $-a_{1}+a_{3}$ which are not included in (\ref{eqn3.7}).

Let $-a_{2}+a_{3}=-a_{1}+a_{2}$. Then also we get two extra positive integers $-a_{1}+a_{3}$ and $-a_{1}+a_{2}+a_{3}$ such that
\[
-a_{1}+a_{2}<-a_{1}+a_{3}<a_{1}+a_{2}<a_{1}+a_{3}<-a_{1}+a_{2}+a_{3}<a_{2}+a_{3}.
\]

{\bf Subcase (ii).}
Let $-a_{1}+a_{3}>a_{1}+a_{2}$. Then, we get an extra positive integer $-a_{1}+a_{3}$ such that
\[
a_{1}+a_{2}<-a_{1}+a_{3}<a_{1}+a_{3}.
\]

If $-a_{2}+a_{3}\neq -a_{1}+a_{2}$ and $-a_{2}+a_{3}\neq a_{1}+a_{2}$, then we are done as we get one more extra positive integer $-a_{2}+a_{3}$.

If $-a_{2}+a_{3}=-a_{1}+a_{2}$, then we get an extra positive integer $-a_{1}-a_{2}+a_{3}$ such that
\[
a_{1}-a_{2}<-a_{1}-a_{2}+a_{3}<-a_{1}+a_{2}.
\]

If $-a_{2}+a_{3}=a_{1}+a_{2}$, then also we are done as we get an extra positive integer $-a_{1}-a_{2}+a_{3}$ such that
\[
-a_{1}+a_{2}<-a_{1}-a_{2}+a_{3}<a_{1}+a_{2}.
\]

{\bf Subcase (iii).}
Let $-a_{1}+a_{3}=a_{1}+a_{2}$. If $-a_{2}+a_{3}<-a_{1}+a_{2}$, then we get two extra positive integers $-a_{2}+a_{3}$ and $-a_{1}-a_{2}+a_{3}$ such that
\[
a_{1}-a_{2}<-a_{1}-a_{2}+a_{3}<-a_{2}+a_{3}<-a_{1}+a_{2}.
\]

If $-a_{2}+a_{3}=-a_{1}+a_{2}$, then $a_{2}=3a_{1}$ and $a_{3}=5a_{1}$. We get two extra positive integers $-a_{1}-a_{2}+a_{3}$ and $a_{1}-a_{2}+a_{3}$ such that
\[
a_{1}-a_{2}<-a_{1}-a_{2}+a_{3}<-a_{1}+a_{2}<a_{1}-a_{2}+a_{3}<a_{1}+a_{2}.
\]

Now, let $-a_{2}+a_{3}>-a_{1}+a_{2}$. Then we get an extra positive integer $-a_{2}+a_{3}$ such that
\[
-a_{1}+a_{2}<-a_{2}+a_{3}<-a_{1}+a_{3}=a_{1}+a_{2}.
\]

If $a_{2}-a_{1}\neq a_{1}$, then $-a_{1}+a_{2}+a_{3}\neq a_{1}+a_{3}$. So, we get one more extra positive integer $-a_{1}+a_{2}+a_{3}$ such that
\[
a_{1}+a_{2}=-a_{1}+a_{3}<-a_{1}+a_{2}+a_{3}<a_{2}+a_{3}.
\]

Let $a_{2}-a_{1}=a_{1}$. So, $a_{2}=2a_{1}$ and $a_{3}=4a_{1}$. If $a_{4}-a_{3}>a_{1}$, then we get an  extra positive integer $a_{2}+a_{4}$ such that
\[
a_{1}+a_{2}+a_{3}<a_{2}+a_{4}<a_{1}+a_{2}+a_{4}.
\]

If $a_{4}-a_{3}<a_{1}$, then we get an  extra positive integer $a_{2}+a_{4}$ such that
\[
a_{2}+a_{3}<a_{2}+a_{4}<a_{1}+a_{2}+a_{3}.
\]

If $a_{4}-a_{3}=a_{1}$, then also we get an extra positive integer $a_{1}-a_{2}+a_{4}$ such that
\[
a_{1}+a_{2}<a_{1}-a_{2}+a_{4}<a_{1}+a_{3}.
\]

Thus, in Cases 1 and  3, we get at least two extra positive integers. As the inverses of these extra integers are also in $3\hat{}_{\pm}A$, so  we get four extra integers in these two cases, which are not present in (\ref{eqn3.7}). In Case 2, we get at least three extra integers. Therefore, in each case we get at least three extra integers in $3\hat{}_{\pm}A$ which are not present in (\ref{eqn3.7}). Hence
\[
|3\hat{}_{\pm}A|\geq 6k-11.
\]
This establishes (\ref{eqn3.6}).

Now, let $|3\hat{}_{\pm}A|=6k-11$. From the above discussion it is clear that we are in Case 2 with   $a_{3}-a_{2}=a_{2}-a_{1}=a_{1}$.

Let $A'=A\setminus \{0\}$. Then, $A'$ is a finite set of $k-1$ positive integers such that $3\hat{}A'\subseteq [a_{1}+a_{2}+a_{3},a_{k-3}+a_{k-2}+a_{k-1}]$. Since $|3\hat{}_{\pm}A|=6k-11$, it follows from the above discussion that $|3\hat{}A'|=3k-11$. Thus, Theorem \ref{thmB} imply the set $A'$ is in arithmetic progression, i.e.,
\[
a_{k-1}-a_{k-2}=a_{k-2}-a_{k-3}=\cdots=a_{2}-a_{1}=d.
\]
Hence, $A=a_{1}*\{0,1,2,\ldots,k-1\}$.\\
This completes the proof of the theorem.
\end{proof}

We observe in the following theorem that the minimum requirement of five elements in the set $A$ in Theorem \ref{thm3.4} is the best possible.

\begin{theorem}\label{thm3.5}
Let $A$ be a set of four non-negative integers with $0\in A$. Then
\begin{equation}\label{eqn3.8}
|3\hat{}_{\pm}A| \geq 12.
\end{equation}
Moreover, if $|3\hat{}_{\pm}A|=12$, then $A=d*\{0,1,2,4\}$.
\end{theorem}

\begin{proof}
Let $A=\{0,a_{1},a_{2},a_{3}\}$, where $0<a_{1}<a_{2}<a_{3}$. From Theorem \ref{thm3.1}, it follows that $3\hat{}_{\pm}A$ contains at least the following ten integers.
\begin{align}\label{eqn3.9}
&-a_{1}-a_{2}-a_{3}<-a_{2}-a_{3}<-a_{1}-a_{3}<-a_{1}-a_{2}<a_{1}-a_{2}<-a_{1}+a_{2}\nonumber\\
&\quad <a_{1}+a_{2}<a_{1}+a_{3}<a_{2}+a_{3}<a_{1}+a_{2}+a_{3}.
\end{align}

Again, from the proof of Theorem \ref{thm3.4}, it follows that the sumset $3\hat{}_{\pm}A$ contains at least three extra integers, except when $a_{2}=2a_{1}$, $a_{3}=4a_{1}$. In the case $a_{2}=2a_{1}$, $a_{3}=4a_{1}$, we get two extra integers. Therefore, we always get two extra integers in $3\hat{}_{\pm}A$ which are not present in (\ref{eqn3.9}). Hence $|3\hat{}_{\pm}A| \geq 12$. This establishes (\ref{eqn3.8}). Moreover, if $|3\hat{}_{\pm}A|=12$, then we have $a_{2}=2a_{1}$ and $a_{3}=4a_{1}$. Hence $A=a_{1}*\{0,1,2,4\}$.
\end{proof}

We finally conjecture the inverse problem as follows:
\begin{conjecture}\label{conj3.2}
Let $A$ be a finite set of $k$ $(\geq 5)$ non-negative integers with $0 \in A$. Let $3\leq h\leq k-1$ be a positive integer. If $|h\hat{}_{\underline{+}}A|=2hk-h(h+1)+1$, then $A=d*[0,k-1]$, for some positive integer $d$.
\end{conjecture}

Theorem \ref{thm3.4} confirms Conjecture \ref{conj3.2} for $h=3$.

\bibliographystyle{amsplain}

\begin{thebibliography}{20}

\bibitem{ANR1} N. Alon, M. B. Nathanson, I. Z. Ruzsa, Adding distinct congruence classes modulo a prime,
 {\it Amer. Math. Monthly}, {\bf 102} (1995) 250--255.

\bibitem{ANR2} N. Alon, M. B. Nathanson, I. Z. Ruzsa, The polynomial method and restricted sums of congruence classes, {\it J. Number Theory}, {\bf 56} (1996) 404--417.

\bibitem{BM15} B. Bajnok, R. Matzke, The minimum size of signed sumsets, {\it Electron. J. Comb.}, {\bf 22} (2) (2015) P2.50.

\bibitem{BM2015} B. Bajnok, R. Matzke, On the minimum size of signed sumsets in elementary abelian groups, {\it J. Number Theory}, {\bf 159} (2015) 384--401.

\bibitem{BR03} B. Bajnok, I. Ruzsa, The independence number of a subset of an abelian group, {\it Integers}, {\bf 3} (2003), Paper No. A2.
    
\bibitem{BP} J. Bhanja, R. K. Pandey, Direct and inverse theorems on signed sumsets of integers, {\it J. Number Theory}, {\bf 196} (2019) 340--352.

\bibitem{cauchy} A. L. Cauchy, Recherches sur les nombres, {\it J. \'{E}cole Polytech.}, {\bf 9} (1813) 99--116.

\bibitem{dav} H. Davenport, On the addition of residue classes, {\it J. Lond. Math. Soc.}, {\bf 10} (1935) 30--32.
\bibitem{SH} J. A. Dias da Silva, Y. O. Hamidoune, Cyclic spaces for Grassmann
derivatives and additive theory, {\it Bull. London Math. Soc.}, {\bf 26} (1994) 140--146.

\bibitem{EK} S. Eliahou, M. Kervaire, Restricted sumsets in finite vector spaces: the case $p=3$,
{\it Integers}, {\bf 1} (2001) Paper No. A02.

\bibitem{EKP} S. Eliahou, M. Kervaire, A. Plagne, Optimally small sumsets in finite abelian
groups, {\it J. Number Theory}, {\bf 101} (2003) 338-–348.

\bibitem{EH} P. Erd\H{o}s, H. Heilbronn, On the addition of residue classes (mod p), {\it Acta Arith.}, {\bf 9} (1964) 149-–159.

\bibitem{freiman} G. A. Freiman, Foundations of Structural Theory of Set Addition, vol. {\bf 37} of Translations of Mathematical Monographs. American Mathematical Society, Provedience, R.I., 1973.

\bibitem{KL03} B. Klopsch, V. F. Lev, How long does it take to generate a group$?$, \textit{J. Algebra}, {\bf 261} (2003) 145--171.

\bibitem{KL09} B. Klopsch, V. F. Lev, Generating abelian groups by addition only, \textit{Forum Math.}, {\bf 21} (2009) 23--41.

\bibitem{mann} H. B. Mann, Addition Theorems: the Addition Theorems of Group Theory and Number Theory, Wiley-Interscience, New York, 1965.

\bibitem{nathu}  M. B. Nathanson, Inverse theorems for subset sums, {\it Trans. Amer. Math.
Soc.}, {\bf 347} (1995) 1409--1418.

\bibitem{N96} M. B. Nathanson, Additive Number Theory: Inverse problems and the geometry of sumsets, Springer, 1996.

\bibitem {plagne} A. Plagne, Optimally small sumsets in groups, I. The supersmall sumset property, the $\mu_{G}^{(k)}$ and the $\nu_{G}^{(k)}$ functions, {\it Unif. Distrib. Theory,\/} {\bf 1} (2006) no. 1, 27--44.

\bibitem {tao} T. Tao, V. H. Vu, Additive Combinatorics, Cambridge studies in advanced mathematics, vol. {\bf 105}, Cambridge, 2010.

\end{thebibliography}

\end{document}